\documentclass[11pt,a4paper]{amsart}

\usepackage{amsmath,amscd}
\usepackage{amssymb}
\usepackage{amsthm}

\usepackage{graphicx}

\usepackage{mathrsfs}
\usepackage[ocgcolorlinks, linkcolor=blue]{hyperref}

\usepackage{calc}
             {\begin{list}{\arabic{enumi}.}{\usecounter{enumi}%
              \setlength{\labelsep}{0.5em}%
              \settowidth{\labelwidth}{(\arabic{enumi})}%
              \setlength{\leftmargin}{\labelwidth+\labelsep}}}%
             {\end{list}}

\newtheorem{Theorem}{Theorem}[section]

\theoremstyle{definition}

\numberwithin{equation}{section}

\newcommand{\mR}{\mathbb{R}}                    
\newcommand{\abs}[1]{\lvert #1 \rvert}          
\newcommand{\norm}[1]{\lVert #1 \rVert}         

\newcommand{\ol}[1]{\overline{#1}}

\newcommand{\mF}{\mathscr{F}}

\newcommand{\eps}{\varepsilon}

\newcounter{sidenote}
\begin{document}

\title{The fractional Calder{\'o}n problem} 

\author[M. Salo]{Mikko Salo}
\address{University of Jyvaskyla, Department of Mathematics and Statistics, PO Box 35, 40014 University of Jyvaskyla, Finland}
\email{mikko.j.salo@jyu.fi}




\begin{abstract}
We review recent progress in the fractional Calder\'on problem, where one tries to determine an unknown coefficient in a fractional Schr\"odinger equation from exterior measurements of solutions. This equation enjoys remarkable uniqueness and approximation properties, which turn out to yield strong results in related inverse problems.
\end{abstract}

\maketitle

\section{Introduction} \label{sec_introduction}

In this expository note, we will discuss recent results for a fractional version of the inverse problem of Calder\'on. Let $0 < s < 1$, and denote by $(-\Delta)^s$ the fractional Laplacian in $\mR^n$ defined by 
\[
(-\Delta)^s u = \mF^{-1}\{ \abs{\xi}^{2s} \hat{u}(\xi) \}
\]
where $\mF u = \hat{u}$ is the Fourier transform of $u$. Observe that the fractional Laplacian is a nonlocal operator: the support of $(-\Delta)^s u$ can be much larger than the support of $u$, and computing $(-\Delta)^s u(x)$ at some point $x \in \mR^n$ requires knowledge of the values of $u$ in all of $\mR^n$. 

Let $\Omega \subset \mR^n$ be a bounded Lipschitz domain and let $q \in L^{\infty}(\Omega)$. Consider solutions $u \in H^s(\mR^n)$, where $H^s$ denotes the standard $L^2$-based Sobolev space, of the fractional Schr\"odinger equation 
\[
\left\{ \begin{array}{rcl}
((-\Delta)^s + q)u &\!\!\!=\!\!\!& 0 \text{ in $\Omega$}, \\[5pt]
u|_{\Omega_e} &\!\!\!=\!\!\!& f
\end{array} \right.
\]
where $\Omega_e = \mR^n \setminus \overline{\Omega}$ is the exterior domain. We assume that $0$ is not an exterior Dirichlet eigenvalue, i.e., 
\begin{equation} \label{dirichlet_uniqueness}
\left\{ \begin{array}{c} \text{if $u \in H^s(\mR^n)$ solves $((-\Delta)^s + q)u = 0$ in $\Omega$ and $u|_{\Omega_e} = 0$,} \\
\text{then $u \equiv 0$.} \end{array} \right.
\end{equation}
This holds e.g.\ if $q \geq 0$. Then there is a unique solution $u \in H^s(\mR^n)$ for any $f \in H^s(\Omega_e)$ (see e.g.\ \cite{GhoshSaloUhlmann}).

We assume that we have access to measurements of solutions outside $\Omega$. The inverse problem will be determine an unknown potential $q$ in $\Omega$ from these measurements. The boundary measurements will be encoded by the exterior Dirichlet-to-Neumann map (DN map for short),
\[
\Lambda_q: H^s(\Omega_e) \to H^s(\Omega_e)^*
\]
that maps $f$ to a nonlocal analogue of the Neumann boundary value of the solution $u$. Formally $\Lambda_q f = (-\Delta)^s u|_{\Omega_e}$. (See \cite{GhoshSaloUhlmann} for a more precise treatment, also in the case where $\Omega$ is a general bounded open set.)

The following result states that exterior measurements, even on arbitrary, possibly disjoint subsets of $\Omega_e$, uniquely determine the potential in $\Omega$.

\begin{Theorem}{\cite{GhoshSaloUhlmann}} \label{thm_main}
Let $\Omega \subset \mR^n$, $n \geq 1$, be bounded open, let $0 < s < 1$, and let $q_1, q_2 \in L^{\infty}(\Omega)$ satisfy \eqref{dirichlet_uniqueness}. Let also $W_1, W_2 \subset \Omega_e$ be open. If the DN maps for the equations $((-\Delta)^s + q_j) u = 0$ in $\Omega$ satisfy 
\[
\Lambda_{q_1} f|_{W_2} = \Lambda_{q_2} f|_{W_2} \text{ for any $f \in C^{\infty}_c(W_1)$,}
\]
then $q_1 = q_2$ in $\Omega$.
\end{Theorem}

This theorem is a fractional version of uniqueness results in the classical inverse problem of Calder\'on (see \cite{Uhlmann_survey} for many results and references), where $s=1$ and measurements are taken on $\partial \Omega$. We note that the fractional problem, where $0 < s < 1$, has several interesting features when compared to the standard Calder\'on problem:
\begin{itemize}
\item 
The same method proves Theorem \ref{thm_main} in all dimensions $n \geq 1$, whereas in the standard Calder\'on problem one often needs different methods for $n=2$ and $n \geq 3$ (and uniqueness fails for $n=1$).
\item 
Theorem \ref{thm_main} proves uniqueness with measurements in arbitrarily small, possibly disjoint sets in the exterior. The standard Calder\'on problem with measurements on an arbitrary subset of the boundary is still open in dimensions $n \geq 3$, and the case of disjoint sets may be even more difficult (see \cite{DaudeKamranNicoleau} and references therein).
\item 
The proof is based on remarkable uniqueness and approximation properties of the fractional Schr\"odinger equation (see Section \ref{sec_tools}). These replace the method of complex geometrical optics solutions that is typical in the standard Calder\'on problem.
\end{itemize}

The above facts suggest that the fractional Calder\'on problem is more manageable than the classical problem, and one could hope for a fairly complete understanding of this inverse problem. Heuristically, this is also explained by a formal variable count: one tries to determine a function of $n$ variables (the potential $q$) from data that depends on $2n$ variables (the Schwartz kernel of the exterior DN map $\Lambda_{q}$). This makes the fractional inverse problem formally overdetermined in any dimension $n \geq 1$.

\subsection*{Extensions}

Theorem \ref{thm_main} has already been extended in several directions:
\begin{enumerate}
\item 
{\it Low regularity.} Uniqueness has been proved in \cite{RulandSalo_stability} for a large class of low regularity potentials, including potentials in $L^{\frac{n}{2s}}(\Omega)$ (the scale invariant $L^p$ space for this equation) or potentials in $W^{-s,\frac{n}{s}}(\Omega)$ that vanish near the boundary.
\item 
{\it Stability.} The work \cite{RulandSalo_stability} also gives a quantitative version of Theorem \ref{thm_main}, showing that this inverse problem enjoys logarithmic stability. One of the results in \cite{RulandSalo_stability} states that if $\Omega$ is smooth and if one has the a priori bound $\norm{q_j}_{W^{\delta,\frac{n}{2s}}} \leq M$ for some $\delta > 0$, then 
\[
\norm{q_1-q_2}_{L^{\frac{n}{2s}}(\Omega)} \leq \omega(\norm{\Lambda_{q_1}-\Lambda_{q_2}}_*)
\]
where $\omega$ is a logarithmic modulus of continuity and $\norm{\,\cdot\,}_*$ is the natural norm for the exterior DN map. In \cite{RulandSalo_instability} this type of stability is proved to be optimal, showing that the fractional inverse problem is in general highly ill-posed. \\
\item 
{\it Reconstruction.} Constructive procedures for recovering $q$ from $\Lambda_q$ are presented in \cite{GRSU}, even in the case of a {\it single} measurement (a related result for obstacles is in \cite{CLL}), and in the work \cite{HarrachLin} that involves monotonicity methods and shape reconstruction.
\item 
{\it Anisotropic problem.} The work \cite{GhoshLinXiao} proves a version of Theorem \ref{thm_main}, where the operator $(-\Delta)^s + q$ is replaced by $(-\mathrm{div}(A \nabla \,\cdot\,))^s + q$ where $A \in C^{\infty}(\mR^n, \mR^{n \times n})$ is a given uniformly elliptic matrix function. The corresponding result for $s=1$ is open when $n \geq 3$.
\item 
{\it Semilinear equations.} A version of Theorem \ref{thm_main} that applies to semilinear equations $(-\Delta)^s u + q(x,u) = 0$ is proved in \cite{LaiLin}.
\end{enumerate}

\subsection*{Background}

The study of fractional and nonlocal operators is currently an active research field and the related literature is substantial. We only mention that operators of this type arise in problems involving anomalous diffusion and random processes with jumps, and they have applications in probability theory, physics, finance, and biology. See \cite{BucurValdinoci, RosOton} for further information and references.

The mathematical study of inverse problems for fractional equations goes back at least to \cite{CNYY}. By now there are a number of results, largely for time-fractional models and including many numerical works. Here is an example of the rigorous results that are available \cite{SakamotoYamamoto}: in the time-fractional heat equation 
\[
\partial_t^{\alpha} u - \Delta u = 0 \text{ in $\Omega \times (0,T)$,} \qquad u|_{\partial \Omega \times (0,T)} = 0,
\]
where $0 < \alpha < 1$ and $\partial_t^{\alpha}$ is the Caputo derivative, $u(0)$ is determined by $u(T)$ in a mildly ill-posed way (for $\alpha=1$ this problem is severely ill-posed). In general, nonlocality may influence the nature of the inverse problem but there are several aspects to be taken into account. We refer to \cite{JinRundell} for a detailed discussion and many further references. 

This article is organized as follows. Section \ref{sec_introduction} is the introduction. In Section \ref{sec_tools} we describe the main tools, namely the strong uniqueness and approximation properties of the fractional equation, that are used in solving the inverse problem. Section \ref{sec_proofs} contains sketches of proofs of the main results.

\subsection*{Acknowledgements}

The author is supported by the Academy of Finland (Finnish Centre of Excellence in Inverse Problems Research, grant numbers 284715 and 309963) and an ERC Starting Grant (grant number 307023).

\section{Tools} \label{sec_tools}

The proof of Theorem \ref{thm_main} begins by showing that if the two DN maps are equal, then (exactly as in the usual Calder\'on problem) one has the integral identity 
\[
\int_{\Omega} (q_1-q_2) u_1 u_2 \,dx = 0
\]
for any $u_j \in H^s(\mR^n)$ that solve $((-\Delta)^s + q_j) u_j = 0$ in $\Omega$ and satisfy $\mathrm{supp}(u_j) \subset \overline{\Omega} \cup \overline{W}_j$. For the standard Schr\"odinger equation, one then typically uses special complex geometrical optics solutions $u_j$ to show that the products $\{Êu_1 u_2 \}$ form a complete set in $L^1(\Omega)$. See \cite{Uhlmann_survey} for an overview.

However, solutions of the fractional Schr\"odinger equation are much less rigid than those of the usual Schr\"odinger equation. The fractional equation enjoys stronger uniqueness and approximation properties:

\begin{Theorem}{\cite{GhoshSaloUhlmann}} \label{thm_uniqueness}
If $0 < s < 1$, if $u \in H^{-r}(\mR^n)$ for some $r \in \mR$, and if both $u$ and $(-\Delta)^s u$ vanish in some open set, then $u \equiv 0$.
\end{Theorem}

\begin{Theorem}{\cite{GhoshSaloUhlmann}} \label{thm_approximation}
Let $\Omega \subset \mR^n$ be a bounded open set, and let $\Omega_1 \subset \mR^n$ be any open set with $\Omega \subset \Omega_1$ and $\Omega_1 \setminus \overline{\Omega} \neq \emptyset$.
\begin{itemize}
\item[(a)] 
If $q \in L^{\infty}(\Omega)$ satisfies \eqref{dirichlet_uniqueness}, then any $f \in L^2(\Omega)$ can be approximated arbitrarily well in $L^2(\Omega)$ by functions $u|_{\Omega}$ where $u \in H^s(\mR^n)$ satisfy 
\[
((-\Delta)^s + q) u = 0 \text{ in $\Omega$}, \qquad \mathrm{supp}(u) \subset \ol{\Omega}_1.
\]
\item[(b)] 
If $\Omega$ has $C^{\infty}$ boundary, and if $q \in C^{\infty}_c(\Omega)$ satisfies \eqref{dirichlet_uniqueness}, then any $f \in C^{\infty}(\ol{\Omega})$ can be approximated arbitrarily well in $C^{\infty}(\ol{\Omega})$ by functions $d(x)^{-s} u|_{\Omega}$ where $u \in H^s(\mR^n)$ satisfy 
\[
((-\Delta)^s + q) u = 0 \text{ in $\Omega$}, \qquad \mathrm{supp}(u) \subset \ol{\Omega}_1.
\]
(Here $d$ is any function in $C^{\infty}(\ol{\Omega})$ with $d(x) = \mathrm{dist}(x,\partial \Omega)$ near $\partial \Omega$ and $d > 0$ in $\Omega$. Also, $v_j \to v$ in $C^{\infty}(\ol{\Omega})$ means that $v_j \to v$ in $C^k(\ol{\Omega})$ for all $k \geq 0$.)
\end{itemize}
\end{Theorem}

Note that corresponding results fail for the Laplacian: if $u \in C^{\infty}_c(\mR^n)$ then both $u$ and $\Delta u$ vanish in a large set but $u$ can be nontrivial, and the set of harmonic functions in $L^2(\Omega)$ is a closed subspace of $L^2(\Omega)$ which is smaller than $L^2(\Omega)$.

Theorem \ref{thm_uniqueness} is classical \cite{Riesz} at least with stronger conditions on $u$, and even the strong unique continuation principle holds \cite{FallFelli, Ruland, Yu}. We note that related results appear in the mathematical physics literature in connection with anti-locality and the Reeh-Schlieder theorem, see \cite{Verch}.

A $C^k$ version of the approximation result, Theorem \ref{thm_approximation}, was first proved in \cite{DipierroSavinValdinoci} when $\Omega = B_1$ and $q=0$. We note that a similar strong approximation property holds for a large class of nonlocal equations including the fractional heat and wave equations $(\partial_t + (-\Delta)^s) u = 0$ and $(\partial_t^2 + (-\Delta)^s) u = 0$, see \cite{DipierroSavinValdinoci_nonlocal, RulandSalo_nonlocal}. This suggests that one could treat inverse problems for these equations as well. 

We will give a proof of Theorem \ref{thm_uniqueness} based on the Caffarelli-Silvestre extension \cite{CaffarelliSilvestre}. This allows us to interpret the quantities $u|_W$ and $(-\Delta)^s u|_W$ as the Cauchy data on $W \times \{0\}$ for the solution $w$ of 
\[
\left\{ \begin{array}{rcl}
\mathrm{div}(x_{n+1}^{1-2s} \nabla w) &\!\!\!=\!\!\!& 0 \text{ in $\mR^{n+1}_+$}, \\[5pt]
w|_{\mR^n \times \{0\}} &\!\!\!=\!\!\!& u.
\end{array} \right.
\]
This reduces the proof of Theorem \ref{thm_uniqueness} to a unique continuation statement for this degenerate elliptic equation.

The approximation property, Theorem \ref{thm_approximation}, follows from the uniqueness result using a Runge type argument \cite{Lax, Malgrange}. The $L^2$ approximation result, which is sufficient for proving Theorem \ref{thm_main}, only requires the basic well-posedness theory for fractional Dirichlet problems. However, for the $C^{\infty}$ approximation one needs to invoke the higher regularity theory for these problems \cite{Hormander_unpublished, Grubb}.

\section{Proofs} \label{sec_proofs}
 
We will first sketch the proof of Theorem \ref{thm_main}, which follows easily from the $L^2$ approximation property in Theorem \ref{thm_approximation}(a).

\begin{proof}[Proof of Theorem \ref{thm_main}]
We begin with an integral identity proved in \cite{GhoshSaloUhlmann}: one has 
\begin{equation} \label{eq_integral_identity}
( (\Lambda_{q_1} - \Lambda_{q_2} ) f_1, f_2 )_{\Omega_e} = \int_{\Omega} (q_1-q_2) u_1 u_2 \,dx
\end{equation}
whenever $u_j \in H^s(\mR^n)$ satisfy $((-\Delta)^s + q_j)u_j = 0$ in $\Omega$ with $u_j|_{\Omega_e} = f_j$. This is basically an integration by parts formula based on the definition of the exterior DN map $\Lambda_q$ (the left hand side is a natural dual pairing in $\Omega_e$).

If $\Lambda_{q_1} f|_{W_2} = \Lambda_{q_2} f|_{W_2}$ for all $f \in C^{\infty}_c(W_1)$, then \eqref{eq_integral_identity} implies that 
\begin{equation} \label{eq_orthogonality_identity}
\int_{\Omega} (q_1-q_2) u_1 u_2 \,dx = 0
\end{equation}
for all $u_j \in H^s(\mR^n)$ solving $((-\Delta)^s + q_j)u_j = 0$ in $\Omega$ with $u_j|_{\Omega_e} \in C^{\infty}_c(W_j)$. It is thus enough to show that the products $\{ u_1 u_2|_{\Omega} \}$ of such solutions form a complete set in $L^1(\Omega)$. This is a consequence of Theorem \ref{thm_approximation}(a): one can for instance fix any $v \in L^2(\Omega)$ and choose solutions $u_j^{(k)}$ satisfying $u_j^{(k)}|_{\Omega_e} \in C^{\infty}_c(W_j)$ (by the proof of Theorem \ref{thm_approximation}(a) below) such that 
\begin{gather*}
u_1^{(k)} \to v \text{ in $L^2(\Omega)$}, \\
u_2^{(k)} \to 1 \text{ in $L^2(\Omega)$},
\end{gather*}
as $k \to \infty$. Inserting these solutions in \eqref{eq_orthogonality_identity} and letting $k \to \infty$ gives 
\[
\int_{\Omega} (q_1-q_2) v \,dx = 0.
\]
Since $v \in L^2(\Omega)$ was arbitrary, it follows that $q_1 = q_2$.
\end{proof}

It is a natural question to try to relax the assumption $q_j \in L^{\infty}(\Omega)$. In fact, this was done in \cite{RulandSalo_stability} using a version of Theorem \ref{thm_approximation}(a) that gives an approximation result in $H^s(\Omega)$ rather than in $L^2(\Omega)$.

We will next prove the approximation result, Theorem \ref{thm_approximation}, using the uniqueness result (Theorem \ref{thm_uniqueness}). The proof is a standard functional analysis argument, which essentially boils down to computing the formal adjoint of the Poisson operator $P_q$.

\begin{proof}[Proof of Theorem \ref{thm_approximation}]
We give the proof of part (a) in the case where $\Omega$ is a bounded Lipschitz domain (for the case of general open sets see \cite{GhoshSaloUhlmann}).

Let $W$ be a ball such that $\overline{W} \subset \Omega_1 \setminus \overline{\Omega}$.
Let $P_q: C^{\infty}_c(W) \to H^s(\mR^n)$ be the Poisson operator that maps an exterior Dirichlet value $f \in C^{\infty}_c(W)$ to the solution $u \in H^s(\mR^n)$ of $((-\Delta)^s+q) u = 0$ in $\Omega$ satisfying $u|_{\Omega_e} = f$. Define the space 
\[
\mathcal{R} = \{ P_q f|_{\Omega} \,;\, f \in C^{\infty}_c(W) \}.
\]
The result will follow if we can show that $\mathcal{R}$ is a dense subspace of $L^2(\Omega)$.

By the Hahn-Banach theorem, it is enough to prove that any $F \in L^2(\Omega)$ that satisfies $(F, P_q f|_{\Omega})_{L^2(\Omega)} = 0$ for all $f \in C^{\infty}_c(W)$ must satisfy $F \equiv 0$. To do this, let $\varphi \in H^s(\mR^n)$ solve 
\begin{equation} \label{eq_varphi_equation}
((-\Delta)^s + q) \varphi = F \text{ in $\Omega$}, \qquad \varphi|_{\Omega_e} = 0.
\end{equation}
Also extend $q$ by zero from $\Omega$ to $\mR^n$. Then, for any $f \in C^{\infty}_c(W)$, we have 
\begin{align*}
0 &= (F, P_q f|_{\Omega})_{L^2(\Omega)} = (((-\Delta)^s + q) \varphi|_{\Omega}, P_q f - f|_{\Omega})_{L^2(\Omega)}. 
\end{align*}

We may extend the last pairing to $\mR^n$: since $P_q f - f$ is a function in $H^s(\mR^n)$ and is supported in $\overline{\Omega}$, there are $\psi_j \in C^{\infty}_c(\Omega)$ with $\psi_j \to P_q f - f$ in $H^s(\mR^n)$ (see e.g.\ \cite[Theorem 3.29]{McLean}). Thus 
\begin{align*}
0 &= \lim_{j \to \infty} (((-\Delta)^s + q) \varphi, \psi_j)_{L^2(\mR^n)} \\
 &= (((-\Delta)^s + q) \varphi, P_q f - f)_{H^{-s}(\mR^n) \times H^s(\mR^n)}.
\end{align*}
Also $\varphi \in H^s(\mR^n)$ is supported in $\overline{\Omega}$, and we may integrate by parts to show that $ (((-\Delta)^s + q) \varphi, P_q f)_{H^{-s}(\mR^n) \times H^s(\mR^n)} = 0$ since $P_q f$ is a solution in $\Omega$.

It follows that 
\[
0 = -((-\Delta)^s \varphi, f)_{H^{-s}(\mR^n) \times H^s(\mR^n)}
\]
for all $f \in C^{\infty}_c(W)$. But now $\varphi \in H^s(\mR^n)$ satisfies 
\[
\varphi|_W = (-\Delta)^s \varphi|_W = 0.
\]
Theorem \ref{thm_uniqueness} implies that $\varphi \equiv 0$, and consequently also $F \equiv 0$. This proves part (a).

To show part (b), i.e.\ $C^{\infty}$ approximation, the function $F$ in the above proof becomes a very irregular distribution. Then one essentially needs to solve \eqref{eq_varphi_equation} in negative order Sobolev spaces associated with the fractional equation. By duality, this can be reduced to the higher regularity theory for fractional exterior Dirichlet problems \cite{Hormander_unpublished, Grubb}. We refer to \cite{GhoshSaloUhlmann} for the details.
\end{proof}

We mention that \cite{RulandSalo_stability} proves a quantitative version of Theorem \ref{thm_approximation}(a): given $v \in L^2(\Omega)$ and $\eps > 0$, one estimates the size of a \emph{control function} $f$ in $\Omega_e$ such that the corresponding solution $u$ satisfies $\norm{u|_{\Omega} - v}_{L^2(\Omega)} \leq \eps$. This is related to the notion of \emph{cost of (approximate) controllability} in the control theory literature. The proof of the quantitative approximation theorem is based on a quantitative version of the uniqueness result, Theorem \ref{thm_uniqueness}, and a functional analysis argument as in \cite{Robbiano}. A similar argument was used to quantify the classical Runge approximation property for second order elliptic equations in \cite{RulandSalo_Runge}, and also to quantify the approximation property for more general nonlocal equations such as the fractional heat and wave equation \cite{RulandSalo_nonlocal}.

Finally, let us sketch a proof of the uniqueness result, Theorem \ref{thm_uniqueness}, in the spirit of the quantitative proof given in \cite{RulandSalo_stability}. As mentioned above, this is based on the Caffarelli-Silvestre extension \cite{CaffarelliSilvestre}: for any $u \in H^s(\mR^n)$, one can realize $(-\Delta)^s u$ as the limit (with convergence in $H^{-s}(\mR^n)$) 
\[
(-\Delta)^s u = c_s \lim_{x_{n+1} \to 0^+} x_{n+1}^{1-2s} \partial_{n+1} w,
\]
where $w$ solves the Dirichlet problem  
\[
\left\{ \begin{array}{rcl}
\mathrm{div}(x_{n+1}^{1-2s} \nabla w) &\!\!\!=\!\!\!& 0 \text{ in $\mR^{n+1}_+$}, \\[5pt]
w|_{\mR^n \times \{0\}} &\!\!\!=\!\!\!& u.
\end{array} \right.
\]
If $s = 1/2$ this is just the Dirichlet problem for the Laplace equation in $\mR^{n+1}_+$, and $w$ is the harmonic extension of $u$. For a general $s$ with $0 < s < 1$, the weight $x_{n+1}^{1-2s}$ is a Muckenhoupt $A_2$ weight, and the equation is a degenerate elliptic equation that has been studied in \cite{FabesKenigSerapioni, FabesJerisonKenig, CabreSire}.

The main point is that Theorem \ref{thm_uniqueness}, which is a uniqueness statement for the nonlocal operator $(-\Delta)^s$, becomes a \emph{unique continuation statement} for solutions of the local equation $\mathrm{div}(x_{n+1}^{1-2s} \nabla w) = 0$. To prove Theorem \ref{thm_uniqueness}, it is enough to show that if the Cauchy data of $w$ vanish on $W \times \{0\}$ (meaning that $w|_{W \times \{0\}} = \lim_{x_{n+1} \to 0^+} x_{n+1}^{1-2s} \partial_{n+1} w|_{W \times \{0\}} = 0$), then the solution $w$ is identically zero in $\mR^{n+1}_+$. This of course implies that $u \equiv 0$.

The work \cite{RulandSalo_stability} gave a quantitative unique continuation statement of this type for the degenerate elliptic equation $\mathrm{div}(x_{n+1}^{1-2s} \nabla w) = 0$. This was based on Carleman estimates and propagation of smallness, or Lebeau-Robbiano interpolation inequality, arguments (see \cite{AlessandriniRondiRossetVessella, LebeauRobbiano, LeRousseauLebeau}).

The proof proceeds in three steps, which are sketched in the following (see \cite{RulandSalo_stability} for the details). The argument is also illustrated in Figure \ref{fig:prop:small}, which is from \cite{RulandSalo_stability}.

\begin{figure}[t]
\includegraphics[width=\textwidth]{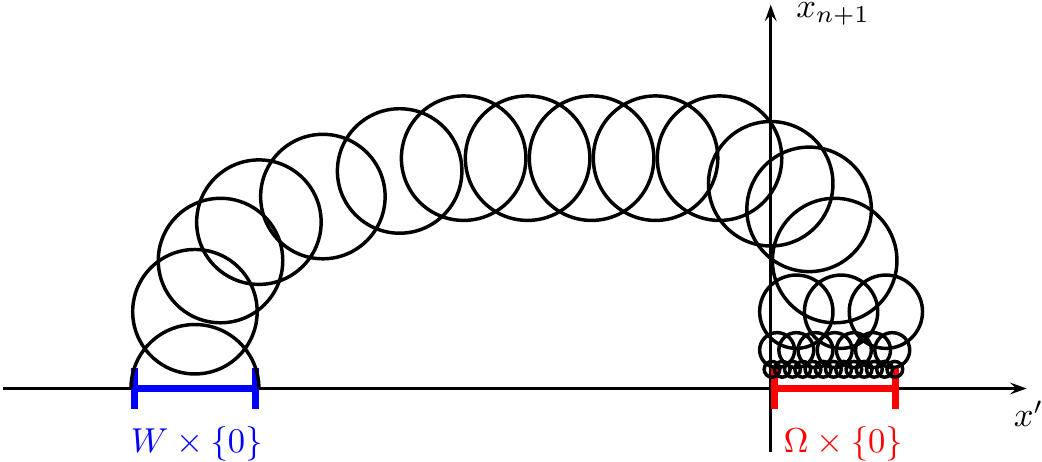}
\caption{An illustration of the propagation of smallness argument.}
\label{fig:prop:small}
\end{figure}

\begin{enumerate}
\item 
If the Cauchy data of $w$ is small in $W \times \{0\}$, then $w$ is small in $W \times (0,1)$. This is proved using a boundary interpolation inequality, which in turn follows from a suitable Carleman inequality with boundary terms.
\item 
If $w$ is small in $W \times (0,1)$, then $w$ is small in $\Omega \times (h,1)$ where $h > 0$ will be specified later. To show this, one propagates the smallness of $w$ in the interior by a chain of balls argument and three balls inequalities, which can again be obtained from a suitable Carleman inequality. Since the balls in the argument have to lie in $\mR^{n+1}_+$, the balls at height $h$ should have radius $\sim h$, and thus one needs $\sim \abs{\log h}$ balls in the chain.
\item 
If $w$ is small in $\Omega \times (h,1)$, then $w$ is small on $\Omega \times \{0\}$. To show this, one first uses a localized trace theorem to estimate the boundary value of $w$ on $\Omega \times \{0\}$ in terms of the size of $w$ in $\Omega \times (0,1)$. One has an estimate in $\Omega \times (h,1)$ from step (2). The $L^2$ norm of $w$ in $\Omega \times (0,h)$ is bounded by a higher $L^p$ norm times $h^{\alpha}$ for some $\alpha > 0$ by the H\"older inequality, and the higher $L^p$ norm of $w$ can be bounded by a $L^2$ norm of $\nabla w$ using a Sobolev embedding for the degenerate equation. Optimizing over $h > 0$ gives the final estimate for $w$ in $\Omega \times \{0\}$ in terms of the Cauchy data on $W \times \{0\}$ and an a priori bound for a weighted $H^1$ norm of $w$ in $\mR^{n+1}_+$.
\end{enumerate}


\bibliographystyle{alpha}

\end{document}